\documentclass[11 pt]{amsart}

\usepackage{amsfonts}
\usepackage{amssymb}
\usepackage{amsmath}
\usepackage{latexsym}
\usepackage{mathrsfs}
\usepackage{amsthm}
\usepackage{verbatim}

\newtheorem{thrm}{Theorem}[section]
\newtheorem{lem}[thrm]{Lemma}
\newtheorem{cor}[thrm]{Corollary}
\newtheorem{prop}[thrm]{Proposition}

\theoremstyle{definition}
\newtheorem{defn}[thrm]{Definition}

\newtheorem{exmple}[thrm]{Example}
\newtheorem{rmk}[thrm]{Remark}

\newtheorem{ques}[thrm]{Question}

\newtheorem{constr}[thrm]{Construction}

\begin{document}

\newcommand{\vol}{\mathrm{vol}}

\newcommand{\Supp}{\mathrm{Supp}}

\newcommand{\Sing}{\mathrm{Sing}}

\newcommand{\ord}{\mathrm{ord}}

\newcommand{\mult}{\mathrm{mult}}

\newcommand{\Val}{\mathrm{Val}}

\title{On Eckl's Pseudo-effective Reduction Map}
\author{Brian Lehmann}
\thanks{This material is based upon work supported under a National Science
Foundation Graduate Research Fellowship.}
\address{Department of Mathematics, Rice University \\
Houston, TX \, \, 77005}
\email{blehmann@rice.edu}

\begin{abstract}
Suppose that $X$ is a complex projective variety and $L$ is a pseudo-effective divisor.  A numerical reduction map is a quotient of $X$ by all subvarieties along which $L$ is numerically trivial.  We construct two variants: the $L$-trivial reduction map and the pseudo-effective reduction map of \cite{eckl05}.  We show that these maps capture interesting geometric properties of $L$ and use them to analyze abundant divisors.
\end{abstract}

\maketitle

\section{Introduction}

Let $X$ be a complex projective variety and let $L$ be a pseudo-effective divisor on $X$.  One fruitful way to study the geometry of $L$ is to analyze maps defined by sections of multiples of $L$.  Our goal is to develop a similar theory of maps that reflect numerical properties of $L$.

The theory of reduction maps, introduced by Tsuji and \cite{8authors}, provides a systematic way of defining such morphisms.  A reduction map is a quotient of $X$ by all subvarieties $V$ for which $L|_{V}$ is ``trivial'' in some sense.  One hopes that $L$ is then ``trivial'' along the fibers of the resulting map, so that we have constructed a universal quotient with respect to this property.

The first numerical reduction maps were developed by \cite{8authors} which worked with nef divisors.  \cite{bdpp04} and \cite{eckl05} considered how to generalize their work to arbitrary pseudo-effective divisors.  However, there are some subtle issues in their arguments.  We extend and correct their ideas (using different techniques) to define two maps: the $L$-trivial reduction map and the pseudo-effective reduction map of \cite{eckl05}.

\subsection{$L$-trivial reduction map} 

The $L$-trivial reduction map is constructed by taking a quotient of $X$ by subvarieties $V$ satisfying $L|_{V} \equiv 0$.   As suggested by \cite{bdpp04}, its key property is that for a general fiber $F$ the numerical dimension of $L|_{F}$ vanishes.  Recall that the numerical dimension $\nu(L)$ is a numerical measure of the positivity of $L$ (see Definition \ref{numdimdefn}).

\begin{thrm} \label{ltrivialreduction}
Let $X$ be a normal projective variety and let $L$ be a pseudo-effective divisor on $X$.
There is a birational model $\phi: Y \to X$ and a surjective morphism $\pi: Y \to Z$ with connected fibers satisfying
\begin{enumerate}
\item For any curve $C$ on $Y$ passing through a very general point with $\dim(\pi(C)) > 0$ we have $\phi^{*}L \cdot C > 0$.
\item For a general fiber $F$ of $\pi$ we have $\nu(\phi^{*}L|_{F}) = 0$.
\item The pair $(Y,\pi)$ is the minimal quotient satisfying (1): if $\widehat{\phi}: \widehat{Y} \to X$
is a birational map and $\widehat{\pi}: \widehat{Y} \to \widehat{Z}$ is a surjective morphism with connected fibers satisfying (1), then there is a dominant rational map $\psi: Z \dashrightarrow \widehat{Z}$ such that $\widehat{\pi}$ is birationally equivalent to $\psi \circ \pi$.
\end{enumerate}
The pair $(Y,\pi)$ is determined up to birational equivalence and only depends on the numerical class of $L$.  The induced map $\pi \circ \phi^{-1}: X \dashrightarrow Z$ is called the \emph{$L$-trivial reduction map}.
\end{thrm}

\begin{rmk}
Early versions of \cite{bdpp04} stated an incorrect variant of Theorem \ref{ltrivialreduction}.  As demonstrated by \cite{bdpp04} Example 8.9, it is not possible to ensure that $\pi \circ \phi: X \dashrightarrow Z$ is almost proper.
\end{rmk}

An important subtlety is that the proper quotient of \cite{campana81} is not the right choice for constructing the map.  \cite{bdpp04} Example 8.9 shows that the proper quotient of $X$ by subvarieties $V$ satisfying $L|_{V} \equiv 0$ may not reflect the birational invariants of $L$.  Instead, one must take a generic quotient in the sense of Section \ref{genericquotientmapssection}.  

\subsection{Pseudo-effective reduction map}
The pseudo-effective reduction map is more natural from the viewpoint of birational geometry.  It satisfies a slightly weaker form of triviality along the general fiber, but it is the maximal quotient with this property.  The pseudo-effective reduction map was first constructed in \cite{eckl05} using analytic techniques; we will give an algebraic formulation clarifying some of the subtleties in the arguments of \cite{eckl05}.

Recall that for a pseudo-effective divisor $L$, the positive part $P_{\sigma}(L)$ is defined by subtracting the ``fixed part'' of $L$ supported on the restricted base locus of $L$.

\begin{thrm}[cf.~\cite{eckl05}, Proposition 4.2] \label{psefreduction}
Let $X$ be a normal projective variety and let $L$ be a pseudo-effective $\mathbb{R}$-divisor on $X$.  There is a birational model $\phi: Y \to X$ and a surjective morphism $\pi: Y \to Z$ with connected fibers satisfying
\begin{enumerate}
\item For a general fiber $F$ of $\pi$ we have $P_{\sigma}(\phi^{*}L)|_{F} \equiv 0$.
\item The pair $(Y,\pi)$ is the maximal quotient satisfying (1): if $\phi': Y' \to X$ is a birational map
and $\pi': Y' \to Z'$ is a surjective morphism with connected fibers satisfying (1), then there
is a dominant rational map $\psi: Z' \dashrightarrow Z$ such that $\pi$ is birationally equivalent to $\psi \circ \pi'$.
\end{enumerate}
The pair $(Y,\pi)$ is determined up to birational equivalence and depends only on the numerical class of $L$.  The induced map $\pi \circ \phi^{-1}: X \dashrightarrow Z$ is called the pseudo-effective reduction map associated to $L$.
\end{thrm}

The pseudo-effective reduction map satisfies the following properties.
\begin{itemize}
\item If $\kappa(L) \geq 0$ then the Iitaka fibration for $L$ factors birationally through the pseudo-effective reduction map.  Thus, the pseudo-effective reduction map identifies a ``numerical piece'' of the Iitaka fibration that is easier to understand.
\item The pseudo-effective reduction map is compatible with birational transformations and the divisorial Zariski decomposition.
\item If $\kappa(L) \geq 0$, there is a morphism $g: W \to T$ birationally equivalent to the pseudo-effective reduction map and a divisor $D$ on $T$ that captures all the interesting geometry of $L$.  In particular there is some integer $m$ so that there is a natural identification of section rings $R(T,mD) = R(X,mL)$.
\end{itemize}

As with other reduction maps, the pseudo-effective reduction map has a nice description in terms of the curves that it contracts.  Given a subvariety $V$ of $X$, one can define a restricted numerical dimension $\nu_{X|V}(L)$ measuring the positivity of $L$ along $V$.

\begin{thrm} \label{explicitpsefreduction}
Let $X$ be a normal projective variety and let $L$ be a pseudo-effective $\mathbb{R}$-divisor on $X$.  The pseudo-effective reduction map for $L$ is the generic quotient of $X$ by all curves $C$ satisfying $\nu_{X|C}(L) = 0$.
\end{thrm}

\begin{rmk}
Many steps in the construction of these numerical reduction maps work in arbitrary characteristic.  The main obstacle to working over uncountable fields of characteristic $p$ is that the numerical dimension has not been studied in this context.
\end{rmk}

\subsection{Abundant divisors}
The pseudo-effective reduction map is most useful for studying divisors whose numerical properties closely reflect the asymptotic behavior of sections.  

\begin{defn}
Let $X$ be a normal variety and let $L$ be a pseudo-effective divisor.  We say that $L$ is abundant if $\kappa(L) = \nu(L)$.
\end{defn}

Using the pseudo-effective reduction map we will show that abundance is equivalent to several other geometric notions.  In particular we generalize the well-known result for nef divisors in \cite{kawamata85}.

\begin{thrm}
Let $X$ be a normal variety and let $L$ be a pseudo-effective divisor.  The following are equivalent:
\begin{enumerate}
\item $L$ is abundant.
\item There is a birational map $\mu: W \to X$ and a morphism $g: W \to T$ such that $P_{\sigma}(\mu^{*}L) \sim_{\mathbb{Q}} P_{\sigma}(g^{*}B)$ for some big divisor $B$ on $T$.
\end{enumerate}
\end{thrm}

Using this criterion, we will show that the numerical invariants of an abundant divisor $L$ are closely related to the asymptotic behavior of the spaces $H^{0}(X,\mathcal{O}_{X}(mL))$ as $m$ increases.  Furthermore, abundant divisors satisfy vanishing theorems similar to those for big divisors.

I would like to thank my advisor J. M\textsuperscript{c}Kernan for his advice and support and Y.~Gongyo for some helpful conversations.

\section{Background} \label{backgroundsection}

All schemes will be of finite type and lie over the base field $\mathbb{C}$.  A variety denotes an irreducible reduced scheme.   All varieties are assumed to be normal and projective unless otherwise qualified.

\subsection{Definitions}

The term ``divisor'' will always refer to an $\mathbb{R}$-Cartier $\mathbb{R}$-Weil divisor.  We use the notations $\sim$, $\sim_{\mathbb{Q}}$, $\sim_{\mathbb{R}}$, and $\equiv$ to denote respectively linear equivalence, $\mathbb{Q}$-linear equivalence, $\mathbb{R}$-linear equivalence, and numerical equivalence of $\mathbb{R}$-divisors.  The $\mathbb{R}$-stable base locus of a divisor $L$ is the Zariski-closed set
\begin{equation*}
\mathbf{B}_{\mathbb{R}}(L) := \bigcap \{ \; \Supp(C) \; | \; C \geq 0 \; \; \textrm{and} \; \; C \sim_{\mathbb{R}} L \}.
\end{equation*}
By convention when $L$ is not $\mathbb{R}$-linearly equivalent to an effective divisor we set $\mathbf{B}_{\mathbb{R}}(L) = X$.  Since we will be interested in perturbations by an ample divisor $A$, we will usually work with a perturbed variation of the $\mathbb{R}$-stable base locus:

\begin{defn}
Let $X$ be a normal variety and let $L$ be a pseudo-effective divisor.  The \emph{restricted base locus} is defined to be
\begin{equation*}
\mathbf{B}_{-}(L) := \bigcup_{A \textrm{ ample}} \mathbf{B}_{\mathbb{R}}(L + A).
\end{equation*}
\end{defn}

It is shown in \cite{nakayama04} that the restricted base locus is a countable union of Zariski-closed subsets of $X$ that only depends on the numerical class of $L$.

We will often construct maps that are only unique up to birational modifications, which we codify using the following terminology.

\begin{defn}
Suppose that $\pi: X \dashrightarrow Z$ and $\pi': X' \dashrightarrow Z'$ are two dominant rational maps of varieties.  We say that $\pi$ and $\pi'$ are \emph{birationally equivalent} if there are birational maps $\phi: X \dashrightarrow X'$ and $\mu: Z \dashrightarrow Z'$ such that
$\mu \circ \pi = \pi' \circ \phi$ on an open subset of $X$.

Suppose that $f: X \dashrightarrow Z$ and $g: X \dashrightarrow Z'$ are two dominant rational maps.  We say that \emph{$f$ factors birationally through $g$} if there is a dominant rational map $\psi: Z' \dashrightarrow Z$ such that $f = \psi \circ g$.
\end{defn}

It will frequently be useful to find a particularly nice representative of a birational equivalence class of maps.  One option is presented by the following special case of Raynaud's flattening theorem.

\begin{thrm} [\cite{raynaud72}, Ch. 4 Theorem 1] \label{hironakaflattening}
Let $f: X \to Z$ be a surjective morphism of normal varieties.  There is a normal variety $T$ and a birational morphism $\psi: T \to Z$ satisfying the following: let $W$ denote the unique (not necessarily normal) component of $X \times_{Z} T$ that dominates $T$ under the projection map.  Then the induced morphism $g: W \to T$ is flat.  Furthermore we may ensure that $\psi$ is an isomorphism on the locus where $f$ is flat.
\end{thrm}

We will always apply this theorem in conjunction with a normalization $\nu: W^{\nu} \to W$ to remain in the category of normal varieties.

Curves will play an essential role in the construction of the pseudo-effective reduction map.  Since
we are primarily interested in what happens near a general point, we will need the following notion.

\begin{defn}
Let $X$ be a variety and let $C$ be an irreducible reduced curve on $X$.  We say that $C$ is \emph{movable} if it is a member of a family of curves dominating $X$.
\end{defn}

Let $\overline{NE}^{1}(X) \subset N^{1}(X)$ denote the closure of the cone generated by classes of effective divisors and $\overline{NM}_{1}(X) \subset N_{1}(X)$ denote the closure of the cone generated by classes of movable curves.  The following theorem of \cite{bdpp04} indicates the fundamental role that movable curves play in birational geometry.

\begin{thrm}[\cite{bdpp04}, Theorem 1.5] \label{bdppmovablecurvetheorem}
Let $X$ be a normal projective variety.  Then $\overline{NM}_{1}(X)$ and $\overline{NE}^{1}(X)$ are dual cones.
\end{thrm}

Although the proof in \cite{bdpp04} assumes that $X$ is smooth, the theorem still holds true in this generality; see \cite{lazarsfeld04} Theorem 11.4.19.

\begin{defn}
Let $f: X \to Z$ be a morphism of normal varieties.  A curve $C$ on $X$ is said to be \emph{$f$-vertical} if $f^{*}H \cdot C = 0$ for some ample divisor $H$ on $Z$.
\end{defn}

\cite{peternell08} proves a relative version of Theorem \ref{bdppmovablecurvetheorem}.

\begin{thrm}[\cite{peternell08}, 6.8 Theorem] \label{relativebdpp}
Let $X$ be a smooth variety and let $f: X \to Z$ be a morphism to a normal variety $Z$.  Let $\overline{NM}_{1}(X/Z)$ denote the closed subcone of $\overline{NM}_{1}(X)$ consisting of curve classes $\alpha$ such that $f^{*}H \cdot \alpha = 0$ for an ample divisor $H$ on $Z$.  Then
\begin{enumerate}
\item If $L$ is a pseudo-effective divisor on $X$, then the restriction $L|_{F}$ is pseudo-effective for a general fiber $F$ of $f$.
\item For a general fiber $F$ of $f$, the injection $i: F \to X$ induces a surjective map $i_{*}: \overline{NM}_{1}(F) \to \overline{NM}_{1}(X/Z)$.
\end{enumerate}
In particular, $\overline{NM}_{1}(X/Z)$ is the closure of the cone generated by $f$-vertical movable curves.
\end{thrm}

\subsection{Divisorial Zariski Decompositions}

For a pseudo-effective divisor $L$ on a surface there is a classical decomposition of $L$ into
a ``positive part'' and a ``negative part'' due to Zariski and Fujita.  In this section we review a generalization to higher dimensions developed by Nakayama (\cite{nakayama04}) in the algebraic setting and by Boucksom (\cite{boucksom04}) analytically.

\begin{defn}
Let $X$ be a normal variety.  Suppose that $L$ is a divisor with $\kappa(L) \geq 0$ and $v$ is a discrete valuation on $X$.  We define the asymptotic order of vanishing of $L$ along $v$ as follows:
\begin{equation*}
v(\Vert L \Vert) = \inf_{\begin{subarray}{l}m \in \mathbb{Z}_{> 0} \\ D \in |mL| \end{subarray}} \frac{1}{m} v(D).
\end{equation*}
\end{defn}

The following description of asymptotic valutions of big divisors was proved in the generality stated here by \cite{elmnp05}.

\begin{thrm}[\cite{elmnp05}, Theorem A] \label{valuationnumericalinvariant}
Let $X$ be a normal variety.  If $L$ is big, then $v(\Vert L \Vert)$ only depends on the numerical class of $L$.  Furthermore, the function $v(\Vert - \Vert)$ is continuous on the big cone.
\end{thrm}

The behavior of asymptotic valuations along the pseudo-effective boundary may be very complicated.  We can define a better-behaved invariant by taking a limit over nearby big divisors.

\begin{defn}
Let $X$ be a smooth variety and let $L$ be a pseudo-effective divisor on $X$.  Fix an ample divisor $A$ on $X$.  For any valuation $v$, we define
\begin{equation*}
\sigma_{v}(L) = \lim_{\epsilon \downarrow 0} v(\Vert L + \epsilon A \Vert).
\end{equation*}
\cite{nakayama04} verifies that this definition is independent of $A$.  For big divisors $\sigma_{v}$ simply recovers $v$ by continuity, so $\sigma_{v}$ is a lower semi-continuous extension of $v$ to the entire pseudo-effective cone. 
\end{defn}

We now restrict our attention to valuations $v_{\Gamma}$ that measure the order of vanishing along a prime divisor $\Gamma$ on $X$.  The key fact about the corresponding function $\sigma_{\Gamma}$ is the following theorem.

\begin{thrm}[\cite{nakayama04},Corollary III.1.11] \label{nsigmafinitelymanycomponents}
Let $X$ be a smooth variety.  For any pseudo-effective divisor $L$ there are only finitely many prime divisors $\Gamma$ with $\sigma_{\Gamma}(L) > 0$.
\end{thrm}

\begin{defn}
Let $X$ be a smooth variety and $L$ be a pseudo-effective divisor. We define $\mathbb{R}$-divisors
\begin{equation*}
N_{\sigma}(L) = \sum \sigma_{\Gamma}(L) \Gamma \qquad \qquad P_{\sigma}(L) = L - N_{\sigma}(L)
\end{equation*}
The decomposition $L = N_{\sigma}(L) + P_{\sigma}(L)$ is called the \emph{divisorial Zariski
decomposition} of $L$.
\end{defn}

The following proposition records the basic properties of the divisorial Zariski decomposition.

\begin{prop} [\cite{nakayama04}, Lemma III.1.4 and Theorem V.1.3]
Let $X$ be a smooth variety and let $L$ be a pseudo-effective divisor.
\begin{enumerate}
\item $N_{\sigma}(L)$ depends only on the numerical class of $L$.
\item $\Supp(N_{\sigma}(L))$ coincides with the divisorial components of $\mathbf{B}_{-}(L)$.
\item $H^{0}(X,\mathcal{O}_{X}(\lfloor mP_{\sigma}(L) \rfloor)) \to H^{0}(X,\mathcal{O}_{X}(\lfloor mL \rfloor))$ is an isomorphism for every $m \geq 0$.
\end{enumerate}
\end{prop}

\subsection{Numerical Dimension}

The numerical dimension $\nu(L)$ is a numerical measure of the positivity of $L$ introduced by \cite{nakayama04} and \cite{bdpp04}.  We will use the definition from \cite{nakayama04}.

\begin{defn} 
Let $X$ be a smooth variety, $L$ a pseudo-effective divisor, and $W$ any reduced closed subset of $X$.  Let $\phi: Y \to X$ be any smooth resolution of $\mathcal{I}_{W}$, and define the divisor $E$ by setting $\mathcal{O}_{X}(-E) = \phi^{-1}\mathcal{I}_{W} \cdot \mathcal{O}_{Y}$.  We say that $L$ numerically dominates $W$, and write $L \succcurlyeq W$, if there exists $\epsilon > 0$ such that $\phi^{*}L - \epsilon E$ is pseudo-effective.  Note that numerical domination is independent of the choice of resolution since the pseudo-effectiveness of a divisor is preserved by pull-backs.
\end{defn}

\begin{rmk} \label{numdimalternateformulation}
It is sometimes useful to have a slightly different formulation.  \cite{lehmann10} Remark 5.2 checks that $\phi^{*}L - \epsilon E$ is not pseudo-effective for any $\epsilon>0$ iff there is some ample divisor $A$ such that  $x\phi^{*}L - E + A$ is not pseudo-effective for any $x \in \mathbb{R}$.
\end{rmk}

\begin{defn} \label{numdimdefn}
Let $X$ be a smooth variety and let $L$ be a pseudo-effective divisor.  We define the numerical dimension $\nu(L)$ as follows.  If $L$ is big, by convention we set $\nu(L) = \dim(X)$.  Otherwise we define
\begin{equation*}
\nu(L) = \min \{ \dim(W) | L \not \succcurlyeq W \}
\end{equation*}
where $W$ varies over all reduced closed subsets of $X$.

If $X$ is normal, we define $\nu(L)$ by first pulling $L$ back to any smooth model of $X$.  Theorem \ref{numdimproperties} shows that this definition is independent of the choice of model.
\end{defn}

It is shown in \cite{lehmann10} that Definition \ref{numdimdefn} is equivalent to the other definitions of the numerical dimension found in \cite{nakayama04} and \cite{bdpp04}.

\begin{thrm}[\cite{nakayama04},Proposition V.2.22 and \cite{lehmann10}, Theorem 6.2]  \label{numdimproperties}
Let $X$ be a smooth variety and let $L$ be a pseudo-effective divisor.
\begin{enumerate}
\item $\nu(L)$ only depends on the numerical class of $L$.
\item $0 \leq \nu(L) \leq \dim(X)$ and $\kappa(L) \leq \nu(L)$.
\item $\nu(L) = \dim(X)$ iff $L$ is big and $\nu(L) = 0$ iff $P_{\sigma}(L) \equiv 0$.
\item $\nu(L) = \nu(P_{\sigma}(L))$.
\item If $\phi: Y \to X$ is a birational morphism from a smooth variety $Y$ then $\nu(\phi^{*}L) = \nu(L)$.
\item Suppose that $f: X \to Z$ has connected fibers and $F$ is a very general fiber of $f$.  Then
$\nu(L) \leq \nu(L|_{F}) + \dim(Z)$.
\end{enumerate}
\end{thrm}

The restricted numerical dimension $\nu_{X|V}(L)$ is a similar notion that measures the positivity of $L$ along a subvariety $V$ of $X$.  It will be most convenient to use a formulation that has a very different feel from Definition \ref{numdimdefn}.

\begin{defn} \label{restnumdimdefn}
Let $X$ be a normal variety,  $L$ a pseudo-effective divisor, and $V$ a subvariety of $X$ such that $V \not \subset \mathbf{B}_{-}(L) \cup \Supp(L) \cup \Sing(X)$.  We define
\begin{equation*}
\nu_{X|V}(L) = \max \{ \dim W | \inf_{\phi} \vol(P_{\sigma}(\phi^{*}L)|_{\widetilde{W}}) >0 \}
\end{equation*}
where $W$ varies over all very general intersections of $V$ with very ample divisors, $\phi$ varies over all birational maps such that no $\phi$-exceptional center contains $W$, and $\widetilde{W}$ denotes the strict transform of $W$.
\end{defn}

Although the restricted numerical dimension is defined by considering all birational models, often it suffices to consider just one fixed sequence of models.

\begin{prop}[\cite{lehmann10}, Proposition 7.9] \label{mapscomputingnumdim}
Let $X$ be a smooth variety and let $L$ be a pseudo-effective divisor.  There is a sequence of birational maps $\phi_{m}: \widetilde{X}_{m} \to X$ such that the following holds.  Suppose that $V$ is any subvariety not contained in $\mathbf{B}_{-}(L) \cup \Supp(L)$.  Then
\begin{equation*}
\nu_{X|V}(L) = \max \{ \dim W | \lim_{m} \vol_{\widetilde{W}_{m}}(P_{\sigma}(\phi_{m}^{*}L)|_{\widetilde{W}_{m}}) >0 \}
\end{equation*}
where $W$ varies over all very general intersections of $V$ with very ample divisors.
\end{prop}

The basic properties of $\nu_{X|V}(L)$ are recorded in the following theorem.

\begin{thrm}[\cite{lehmann10}, Theorem 7.4]  \label{restnumdimproperties}
Let $X$ be a normal variety, $L$ a pseudo-effective divisor, and $V$ a subvariety not contained in $\mathbf{B}_{-}(L) \cup \Supp(L) \cup \Sing(X)$.
\begin{enumerate}
\item $\nu_{X|V}(L)$ only depends on the numerical class of $L$.
\item $\nu_{X|V}(L) \leq \nu(L)$ and  $\nu_{X|V}(L) \leq \nu(L|_{V})$.
\item $\nu_{X|V}(L) = \nu_{X|V}(P_{\sigma}(L))$.
\item If $\phi: Y \to X$ is a birational morphism from a smooth variety $Y$ such that $V$ is not contained in any $\phi$-exceptional center then $\nu_{Y|V'}(\phi^{*}L) = \nu_{X|V}(L)$ where $V'$ is the strict transform of $V$.
\end{enumerate}
\end{thrm}

Finally, we will need a criterion for the vanishing of $\nu(L)$.

\begin{lem}[\cite{lehmann10}, Proposition 7.6] \label{nuvanishingcriterion}
Let $X$ be a smooth variety and let $L$ be a pseudo-effective $\mathbb{R}$-divisor on $X$.  Then $\nu(L) \geq 1$ iff there is a curve $C$ that is a very general complete intersection of very ample divisors such that $\nu_{X|C}(L) = 1$.
\end{lem}

\section{Generic Quotients} \label{genericquotientmapssection}

Let $X$ be a normal variety.  Suppose that we are given a family of subschemes of $X$, that is, a family of schemes $s: U \to V$ with a morphism $w: U \to X$ realizing the fibers of $s$ as subschemes of $X$.  In general it is not possible to construct a morphism $\pi: X \to Z$ that contracts all the members of the family.  However, as was first demonstrated in \cite{campana81}, under some mild constraints we can define such a quotient on an open subset of $X$.  In this section we review how the quotient theory can be used to construct a ``generic quotient'' by a family of subschemes.

Suppose that we have a finite collection of tuples $\{(U_{i},V_{i},s_{i},w_{i})\}$ where $s_{i}: U_{i} \to V_{i}$ is a family of quasi-projective schemes and $w_{i}: U_{i} \to X$ is a morphism.  Given such data, \cite{kollar96} defines a constructible subset $\langle U_{1},\ldots,U_{k} \rangle^{\sim} \subset X \times X$ such that two points $x_{1},x_{2}$ can be connected by a sequence of members of our families iff $(x_{1},x_{2}) \in \langle U_{1},\ldots,U_{k} \rangle^{\sim}$.  We would like to quotient $X$ by this relation.  There are two versions of the quotient: one for open families and one for proper families.

\begin{thrm}[\cite{kollar96},IV.4.13 and IV.4.16] \label{openquotienttheorem}
Let $X$ be a normal variety.  Suppose that $\{ (U_{i},V_{i},s_{i},w_{i}) \}$ is a finite collection of families of quasi-projective subschemes on $X$ such that either
\begin{itemize}
\item each $w_{i}$ is open and each $s_{i}$ is flat with irreducible fibers, or
\item each $w_{i}$ is proper and each $s_{i}$ is proper with connected fibers.
\end{itemize}
Then there is an open subvariety $X^{0} \subset X$, a quasi-projective variety $Z^{0}$, and a morphism $\pi: X^{0} \to Z^{0}$ with connected fibers such that
\begin{enumerate}
\item $\langle U_{1},\ldots,U_{m} \rangle^{\sim}$ restricts to an equivalence relation on $X^{0}$.
\item For every $z \in Z^{0}$ the closure of $\pi^{-1}(z)$ in $X$ coincides with the closure
of a suitable $\langle U_{1},\ldots,U_{m} \rangle^{\sim}$-equivalence class.
\end{enumerate}
Furthermore, in the proper case $\pi$ is a proper map, so that by condition (2) each fiber of $\pi$ coincides with an equivalence class.
\end{thrm}

Although the proper version of the quotient map is more commonly used, the open version is better suited to isolate the generic behavior of the families.  Thus, we will define a generic quotient using the open version of the quotient.  This viewpoint is similar to the notion of ``maximal foliations'' in \cite{eckl05}.

We will only construct the generic quotient of families $\{ (U_{i},V_{i},s_{i},w_{i}) \}$ for which the $w_{i}: U_{i} \to X$ are proper and surjective and the $U_{i}$ are irreducible.  These assumptions mean that the generic quotient accurately reflects the geometry of our families.  (In view of our intended application, these requirements are also justified by Example \ref{necessaryconditions}.)

\begin{constr} \label{genericquotientconstruction}
Let $X$ be a normal variety and let $\{ (U_{i},V_{i},s_{i},w_{i}) \}$ be a finite collection of surjective proper families of subschemes such that $U_{i}$ is irreducible.  Let $U$ be the largest open subset of $X$ such that each $w_{i}$ is flat over $U$ and let $V_{i}^{0}$ be the largest open subset of $V_{i}$ so that $s_{i}$ is flat with irreducible fibers over $V_{i}^{0}$.   We then define $U_{i}^{0} =  w_{i}^{-1}(U) \cap s_{i}^{-1}(V_{i}^{0})$.  Since flat maps are open, we can apply Theorem \ref{openquotienttheorem} to the families of quasi-projective schemes $\{ (U_{i}^{0},V_{i}^{0},s_{i}|_{U_{i}^{0}},w_{i}|_{U_{i}^{0}}) \}$.  We call the resulting quotient  $\pi: X \dashrightarrow Z$ the generic quotient of $X$ by the families $U_{i}$.
\end{constr}

\begin{rmk} \label{infinitefamilies}
We can define the generic quotient of an infinite set of families using the following well-known argument.  Suppose we are given an infinite set $\{ (U_{\alpha},V_{\alpha},s_{\alpha},w_{\alpha}) \}_{\alpha \in A}$ of proper surjective families of subschemes such that $U_{\alpha}$ is irreducible.  For a given finite subset $I \subset A$, we let $\pi_{I}: X \dashrightarrow Z_{I}$ denote the generic quotient defined by the corresponding $U_{i}$. We then define the generic quotient of the $U_{\alpha}$ to be any quotient $\pi_{I}$ such that $Z_{I}$ has the smallest possible dimension.  It is easy to check that the resulting map is unique up to birational equivalence.
\end{rmk}

For the remainder of the section we will study the birational properties of the generic quotient.  We begin by defining the notion of a strict transform family.  Let $(U,V,s,w)$ be a proper family of subschemes of $X$.  Suppose that $\phi: Y \to X$ is a birational morphism such that $w(U)$ is not contained in the center of any $\phi$-exceptional divisor.  Let $U'$ be the normalization of the unique component of $U \times_{X} Y$ that dominates $w(U)$.  We have a natural map $s' = s \circ f: U' \to V$ and a projection $w': U' \to Y$.  We define the strict transform family to be $(U',V,s',w')$.

\begin{prop}
Let $X$ be a normal variety and let $\{ (U_{i},V_{i},s_{i},w_{i}) \}$ be a finite collection of surjective proper families of subschemes such that $U_{i}$ is irreducible.  Let $\psi: X' \to X$ be a birational map such that $w_{i}': U_{i}' \to X'$ is flat for each strict transform family $(U_{i}',V_{i},s_{i}',w_{i}')$.  Then the generic quotient of $X$ by the $U_{i}$ is birationally equivalent to the proper quotient of $X'$ by the $U_{i}'$.
\end{prop}

Note that the map $\psi$ in the statement is guaranteed to exist by Theorem \ref{hironakaflattening}.

\begin{proof}
As before we let $U$ be the largest open region of $X$ such that each $w_{i}$ is flat over $U$ and define $U_{i}^{0}$ as in Construction \ref{genericquotientconstruction}.

We start by proving this in a special case: we construct a birational map $\phi: X' \to X$ by flattening each $w_{i}$ in turn using Theorem \ref{hironakaflattening}.  Since flatness is invariant under pull-back, every strict transform family $w_{i}': U_{i}' \to X'$ is flat.  Note that $\phi^{-1}: X \dashrightarrow X'$ is an isomorphism on $U$.  Thus we may interpret the families $U_{i}^{0}$ as lying on $X'$ and the open quotient of $X'$ by the $U_{i}^{0}$ is the same as the generic quotient.

Let $\pi: X' \dashrightarrow Z$ denote the open quotient of $X'$ by the $U_{i}^{0}$ and let $\tau: X' \dashrightarrow T$ denote the proper quotient of $X'$ by the $U_{i}'$.  Since $U_{i}^{0} \subset U_{i}'$, the universal property of the quotient shows that $\tau$ factors birationally through $\pi$.  Conversely, since $w_{i}'$ is flat, we have
\begin{equation*}
w_{i}'^{-1}(\overline{S}) = \overline{w_{i}'^{-1}(S)}
\end{equation*}
for any subset $S \subset X$.  Letting $S$ be a $\langle U_{0}^{0}, \ldots, U_{m}^{0} \rangle$-equivalence class, we see that if a fiber $F$ of $s_{i}': U_{i}' \to V_{i}$ has that $w_{i}'(F)$ intersects $\overline{S}$ then $F$ intersects $\overline{w_{i}'^{-1}(S)}$.  But this latter set is a union of fibers of $s_{i}'$, showing that $w_{i}'(F)$ is contained in $\overline{S}$.  Thus $\pi$ factors birationally through $\tau$, showing they are birationally equivalent.

Finally, we prove the theorem in general.  Suppose that $\psi: X' \to X$ is a birational map as in the statement and $\mu: \widehat{X} \to X'$ is any higher birational model.  Since $w_{i}'$ is flat, the induced map $U_{i}' \times_{X'} \widehat{X} \to \widehat{X}$ is flat (and hence open) on each irreducible component.  However, only one component of $U_{i}' \times_{X'} \widehat{X}$ dominates $\widehat{X}$.  Thus, the strict transform family $\widehat{U}_{i}$ on $\widehat{X}$ is equal to $U_{i}' \times_{X'} \widehat{X}$.  As a consequence, the proper quotient of $U_{i}'$ on $X'$ is the same as the proper quotient of $\widehat{U}_{i}$ on $\widehat{X}$.  In particular, by passing to higher models we see that any two choices of $\psi$ yield the same reduction map.
\end{proof}

An immediate corollary is that the generic quotient is invariant under passing to strict transform families.  Of course, this result does not depend on the finiteness of the collection of families.

\begin{cor} \label{genericquotientbirationalinvariance}
Let $X$ be a normal variety and let $\{ (U_{i},V_{i},s_{i},w_{i}) \}$ be a collection of surjective proper families of subschemes such that $U_{i}$ is irreducible.  Suppose that $\phi: Y \to X$ is a birational map and denote the strict transform families by $(U_{i}',V_{i},s_{i}',w_{i}')$.  The generic quotient of the $U_{i}$ on $X$ is birationally equivalent to the generic quotient of the $U_{i}'$ on $Y$.  
\end{cor}

Another consequence is that the generic quotient is the ``minimal'' possible proper quotient over all strict transform families.

\begin{prop} \label{genericquotientdominance}
Let $X$ be a normal variety and let $\{ (U_{i},V_{i},s_{i},w_{i}) \}$ be a collection of surjective proper families of subschemes such that $U_{i}$ is irreducible.  Let $\pi: X \dashrightarrow Z$ denote the generic quotient.  Suppose that $\phi: Y \to X$ is any birational map and let $\pi': Y \dashrightarrow Z'$ be the proper quotient of $Y$ by the strict transform families $U_{i}'$.  Then $\pi'$ factors birationally through $\pi$.
\end{prop}

\begin{proof}
Let $\psi: X' \to Y$ be a birational map that simulateously flattens the strict transform maps $w': U_{i}' \to Y$.  Note that $\pi'$ is birationally equivalent to the proper quotient associated to the families $(U_{i}' \times_{Y} X', V_{i}, s_{i}' \circ p_{1,i}, p_{2,i})$ on $X'$.   $U_{i}' \times_{Y} X'$ has one component representing the strict transform family but may have additional components.  Thus the proper quotient of $X'$ by the entire families $U_{i}' \times_{Y} X'$ quotients at least as much as the generic quotient.
\end{proof}

\section{$L$-trivial Reduction Map} \label{ltrivialreductionmapsection}

Suppose that $L$ is a pseudo-effective divisor on $X$.  By analogy with the nef reduction map, one could define a reduction map for $L$ by taking the proper quotient of $X$ by all families of movable curves $C$ with $L \cdot C = 0$.  However, as demonstrated in \cite{bdpp04}, the resulting quotient does not have good properties.  A better approach is to take the generic quotient of $X$ by all families of movable curves satisfying $L \cdot C = 0$.   This is of course equivalent to taking the generic quotient of $X$ by all dominant families of subvarieties $V$ such that $L|_{V} \equiv 0$ for a general member $V$.  We will show that this generic quotient yields a geometrically interesting reduction map.

We start with an example from \cite{bdpp04} showing that the proper quotient of $X$ by movable curves $C$ with $L \cdot C = 0$ does not yield a suitable reduction map.  

\begin{exmple}(\cite{bdpp04}, Example 8.9) \label{bdppexample}
In this example we construct a smooth threefold $X$, a divisor $L$ on $X$ with $\kappa(L) = 1$, and a family of movable curves $C$ such that $L \cdot C = 0$ and the proper quotient of $X$ by the family maps to a point.  Since $\kappa(L) = 1$, this quotient ignores interesting geometric information about $L$.  Thus this example shows how the proper quotient can ``contract too much.''

Let $X'$ be the smooth projective bundle
\begin{equation*}
p: \mathbb{P}(\mathcal{O} \oplus \mathcal{O} \oplus \mathcal{O}(-1)) \to \mathbb{P}^{1}.
\end{equation*}
Let $T'$ denote the section $\mathbb{P}(\mathcal{O}(-1))$ of $p$.  Because $T'$ has normal bundle $N_{T'/X'} = \mathcal{O}(-1) \oplus \mathcal{O}(-1)$, we can construct the flop $\phi: X' \dashrightarrow X$ of $T'$.  We define the divisor $L$ on $X$ to be the strict transform of a fiber $F$ on $X'$.  Note that $\kappa(L) = 1$.  The pair $(X,L)$ will provide our example.

Choose any line $C'$ in $F$ that avoids $T'$ and let $C$ denote the strict transform of this line.  Since $\phi$ is an isomorphism on a neighborhood of $C'$, we have $L \cdot C = F \cdot C' = 0$.  We claim that deformations of $C$ define a connecting family of curves on $X$.  Any two points in a strict transform of a fiber of $p$ can be connected by deformations of $C$.  Furthermore, when we deform $C$ so that it intersects the flipped curve $T$, then $C$ breaks up into two components, one of which is $T$ itself.  So, we can traverse between the strict transforms of two different fibers of $p$ via the flipped curve $T$, showing that any point can be connected to any other by deformations of our family.  Thus, the proper quotient by this family of curves contracts $X$ to a point.
\end{exmple}

\subsection{Relative Triviality and the Numerical Dimension}

In order to prove Theorem \ref{ltrivialreduction}, we will need to further develop the theory of the numerical dimension.  The precise result we will need is:

\begin{thrm} \label{relativetriviality}
Let $X$ be a smooth variety and let $L$ be a pseudo-effective divisor on $X$.  Suppose that $f: X \to Z$ is a surjective morphism with connected fibers such that
\begin{enumerate}
\item $L|_{F} \equiv 0$ for a general fiber $F$.
\item There is a subvariety $T \subset L$ dominating $Z$ and not contained in $\mathbf{B}_{-}(L) \cup \Supp(L)$ such that $\nu_{X|T}(L) = 0$.
\end{enumerate}
Then $\nu(L) = 0$.
\end{thrm}

This result will be used in a similar way as Proposition 2.5 of \cite{8authors} in the construction of the nef reduction map.

We first study the properties of morphisms such that $L$ is numerically trivial along the general fiber.  We will prove several lemmas that build toward Corollary \ref{psigmarelativepositivity}.  This corollary is a numerical version of \cite{nakayama04} III.5.2, and we will closely follow his method of proof.

\begin{lem} \label{surfacelemma}
Let $f: S \to C$ be a morphism with connected fibers from a smooth surface $S$ to a smooth curve $C$.  Let $L$ be a pseudo-effective divisor on $S$ such that $L \cdot F = 0$ for a general fiber $F$ of $f$.
\begin{enumerate}
\item If $L \cdot D = 0$ for every $f$-vertical curve $D$, then $L$ is nef.
\item If $L \cdot D \neq 0$ for some $f$-vertical curve $D$ contained in a fiber $F_{0}$, then there is an $f$-vertical curve $G$ contained in $F_{0}$ satisfying $L \cdot G < 0$.
\end{enumerate}
\end{lem}

\begin{proof}
Let $L = P + N$ be the Zariski decomposition of $L$.  Since $P$ is nef and $P \cdot F = 0$, $P$ has vanishing intersection with every $f$-vertical curve.  Note that $N$ is an (effective) $f$-vertical curve since $N \cdot F = L \cdot F = 0$.

We first show (2).  By assumption $N$ must have some components contained in $F_{0}$.  Recall that the self-intersection matrix of the components of $N$ is negative-definite.  In fact, since $f$-vertical curves in different fibers do not intersect, the same is true just for the components contained in $F_{0}$.  Thus, there is an effective curve $G$ supported on $\Supp(N) \cap \Supp(F_{0})$ with $N \cdot G < 0$.  The same argument shows that in (1) we must have $N=0$ so that $L$ is nef.
\end{proof}

\begin{lem}[cf.~\cite{nakayama04} III.5.2] \label{nakayamalemma}
Let $f: X \to Z$ be a morphism from a smooth variety $X$ to a smooth variety $Z$ with connected fibers.  Let $L$ be a pseudo-effective divisor such that $L|_{F} \equiv 0$ for a general fiber $F$ of $f$.  Suppose that $\Theta$ is a divisor on $Z$ such that the restriction of $L$ to the fiber over a general point of $\Theta$ is not numerically trivial.  Then there is some prime divisor $\Gamma$ with $f(\Gamma) = \Theta$ such that $L|_{\Gamma}$ is not pseudo-effective.
\end{lem}

\begin{proof}
Take an intersection of general very ample divisors on $Z$ to obtain a curve $C$ such that $L$ is not numerically trivial along some fibers over $C \cap \Theta$.  Intersect general very ample divisors on $X$ with $f^{-1}(C)$ to obtain a surface $S$ with a map $f: S \to C$.  Note that $S$ and $C$ are smooth.  Applying Lemma \ref{surfacelemma} to $L|_{S}$, we find that there is some $f$-vertical curve $D$ above $C \cap \Theta$ such that $L \cdot D < 0$.  Since $S$ and $C$ are sufficiently general, such curves must cover a codimension $1$ subset above $\Theta$, proving the theorem.
\end{proof}

\begin{cor} \label{psigmarelativepositivity}
Let $f: X \to Z$ be a morphism from a smooth variety $X$ to a smooth variety $Z$ with connected fibers.  Suppose that $L$ is a pseudo-effective divisor such that $L|_{F} \equiv 0$ for a general fiber $F$ of $f$.  Then there is a subset $U \subset Z$ whose complement is a countable union of subvarieties of codimension $2$ such that $P_{\sigma}(L)$ is numerically trivial along the fiber over any point in $U$.
\end{cor}

\begin{proof}
Fix an ample divisor $H$ on $Z$.  There is a countable union of closed subvarieties $S \subset Z$ such that any curve $C$ on $X$ satisfying $f^{*}H \cdot C = 0$ and $P_{\sigma}(L) \cdot C < 0$ has $f(C) \subset S$.  Since $P_{\sigma}(L)|_{\Gamma}$ is pseudo-effective for any prime divisor $\Gamma$ on $X$, Lemma \ref{nakayamalemma} implies that no component of $S$ has codimension $1$.
\end{proof}

We now have the tools to prove Theorem \ref{relativetriviality}.

\begin{proof}[Proof of Theorem \ref{relativetriviality}:]
To show that $\nu(L) = 0$, by Lemma \ref{nuvanishingcriterion} it suffices to show that a curve $C$ which is a very general intersection of very ample divisors satisfies $\nu_{X|C}(L) = 0$.

Let $f': X' \to Z'$ be a flattening of $f$ (with $Z'$ smooth) and let $\phi: W \to X$ be a smooth model lying above $X'$.  We let $g$ denote the map $W \to Z'$.  Note that no $\phi$-exceptional center dominates $Z$ under $f$.  Since $T$ dominates $Z$, it makes sense to talk of the strict transform $T_{W}$ of $T$.   By Theorem \ref{restnumdimproperties} we still have $\nu_{W|T_{W}}(\phi^{*}L) = 0$.  By choosing $C$ sufficiently general, we may assume that it lies in the locus where $\phi^{-1}$ is an isomorphism.  We will write $C_{W}$ for the strict transform of $C$ on $W$.

Choose an infinite sequence of birational maps $\psi_{m}: W_{m} \to W$ computing the numerical dimension as in Proposition \ref{mapscomputingnumdim} and let $\phi_{m}$ denote the map from $W_{m}$ to $X$.  For each $W_{m}$, there is a subset $V_{m} \subset Z'$ that is a countable union of closed subsets of codimension $2$ such that $P_{\sigma}(\phi_{m}^{*}L)$ is numerically trivial when restricted to any fiber of $g \circ \phi_{m}$ not over $V_{m}$.  Note that $\phi(g^{-1}V_{m})$ is a countable union of codimension $2$ subsets in $X$.  By choosing $C$ very general, we may assume that $g(C_{W})$ is disjoint from $\cup_{m} V_{m}$.  We may furthermore assume that $C$ avoids $\mathbf{B}_{-}(P_{\sigma}(L))$ so that no $\phi_{m}$-exceptional divisor intersects $C$.

Define a surface $S$ by taking an intersection of general very ample divisors containing $C_{W}$ with $g^{-1}g(C_{W})$.  Our goal is to show that $\nu_{W|S}(\phi^{*}L) = 0$, which immediately implies
\begin{equation*}
\nu_{X|C}(L) = \nu_{W|C_{W}}(\phi^{*}L) = 0.
\end{equation*}
It suffices to show that $\nu_{W|A}(\phi^{*}L) = 0$ where $A$ is a very general ample divisor on $S$.  In particular, we may assume that the strict transform of $A$ coincides with the pull-back for each of the maps computing the restricted positive product.

We let $S_{m}$ denote the strict transform of $S$ over $\psi_{m}$ and for clarity let $\mu_{m} = \psi_{m}|_{S_{m}}$.  By Proposition \ref{mapscomputingnumdim}, it suffices to show that
\begin{equation*}
\liminf_{m \to \infty} \mu_{m*}\left( P_{\sigma}(\phi_{m}^{*}L)|_{S_{m}} \right) \cdot A =
\liminf_{m \to \infty} \left( P_{\sigma}(\phi_{m}^{*}L)|_{S_{m}} \right) \cdot \mu_{m}^{*}A = 
 0.
\end{equation*}
By construction $P_{\sigma}(\phi_{m}^{*}L)|_{S_{m}}$ is numerically trivial when restricted to any fiber of $g \circ \mu_{m}$.  This implies that $P_{\sigma}(\phi_{m}^{*}L)|_{S_{m}}$ is nef by Corollary \ref{surfacelemma}.  Furthermore, since every $\mu_{m}$-exceptional divisor is $g \circ \mu_{m}$ vertical we have
\begin{equation*}
P_{\sigma}(\phi_{m}^{*}L)|_{S_{m}} =\mu_{m}^{*}\mu_{m*}P_{\sigma}(\phi_{m}^{*}L)|_{S_{m}}
\end{equation*}
and the pushforward $\mu_{m*}P_{\sigma}(\phi_{m}^{*}L)|_{S_{m}}$ is also nef.

Let $T_{S}$ denote the intersection $T_{W} \cap S$.  Since $T_{W} \not \subset \mathbf{B}_{-}(L)$, by making an appropriate choice of $C$ we may assume that $T_{S} := T_{W} \cap S \not \subset \mathbf{B}_{-}(\phi^{*}L)$.  Then $\nu_{W|T_{S}}(\phi^{*}L)$ vanishes since it is no greater than $\nu_{W|T_{W}}(\phi^{*}L)) = 0$.  Let $T_{S,m}$ denote the strict transform of $T_{S}$ under $\mu_{m}$.  Since every $\mu_{m}$-exceptional curve is $g \circ \mu_{m}$-vertical,
\begin{align*}
\liminf_{m \to \infty} \left( P_{\sigma}(\phi_{m}^{*}L)|_{S_{m}} \right) \cdot \mu_{m}^{*}(T_{S}) &
= \liminf_{m \to \infty} \left( P_{\sigma}(\phi_{m}^{*}L)|_{S_{m}} \right) \cdot T_{S,m} \\
&= 0.
\end{align*}
In particular, the limit of the $\mu_{m*}\left( P_{\sigma}(\phi_{m}^{*}L)|_{S_{m}} \right)$ is a nef class that has vanishing intersection with $F$ and the multisection $T_{S}$.  \cite{8authors} Proposition 2.5 then shows that this limit is numerically equivalent to $0$.  Thus, it has vanishing intersection with $A$, finishing the proof.
\end{proof}

\subsection{Definition and Properties}

The main step in the proof of Theorem \ref{ltrivialreduction} is to show that $\nu(L|_{F}) = 0$ for a general fiber $F$.  We will prove this by carefully examining the connection properties of generic quotients and applying Theorem \ref{relativetriviality} inductively.

\begin{defn}
Let $(U,V,s,w)$ be a proper family of subschemes of $X$.  Suppose that $T \subset X$ is a closed subset.  We define the $w$-closure of $T$, denoted $T^{\overline{w}}$, to be the closed subset $w_{i}(s_{i}^{-1}s_{i}(w_{i}^{-1}T))$.  In other words, the $w$-closure of $T$ consists of all the points of $X$ that can be connected to $T$ via a subscheme in our family by one step.
\end{defn}

\begin{lem}
Let $X$ be a smooth variety and let $L$ be a pseudo-effective divisor.  Suppose that that $(U,V,s,w)$ is a proper surjective family of curves such that $U$ is irreducible, $w: U \to X$ is flat, and $L \cdot C =0$ for a general member of our family.  Suppose that $T \subset X$ is a subvariety containing a point $x$ such that
\begin{enumerate}
\item $x$ does not lie in $\mathbf{B}_{-}(L)$, and
\item every component of $w^{-1}(x)$ intersects an irreducible fiber of $s$.
\end{enumerate}
If $\nu(L|_{T}) = 0$, then $\nu(L|_{T_{i}}) = 0$ for any component $T_{i}$ of $T^{\overline{w_{i}}}$.
\end{lem}

\begin{proof}
Since $w$ is flat, every component of $w^{-1}(T)$ intersects $w^{-1}(x)$.  In particular, every component intersects the locus of irreducible fibers of $s$ and is not contained in $w^{-1}\mathbf{B}_{-}(L)$.  In particular $T_{i} \not \subset \mathbf{B}_{-}(L)$, so $L|_{T_{i}}$ is pseudo-effective and we may analyze its numerical dimension.

Fix a component $T'$ of $w^{-1}(T)$ such that $s^{-1}s(T')$ dominates $T_{i}$.  Consider the subvarieties $Z = s(T')$ and $Y = s^{-1}s(T')$.  By blowing up we may suppose that $T'$, $Z$, and $Y$ are all smooth.  Since the map $w|_{T'}$ is surjective by flatness of $w$, we have $\nu(w^{*}L|_{T'}) = 0$.  In particular $\nu_{Y|T'}(w^{*}L|_{Y}) = 0$.  Furthermore, since by assumption $T'$ hits the locus of irreducible fibers of $s$, we know that $w^{*}L$ is numerically trivial along the general fiber of $s$.  Applying Proposition \ref{relativetriviality} we see that $\nu(w^{*}L|_{Y}) = 0$, and thus $\nu(L|_{T_{i}})=0$.
\end{proof}

\begin{prop} \label{bdpptriviality}
Let $X$ be a normal variety and let $L$ be a pseudo-effective divisor such that the generic quotient of $X$ by all movable curves $C$ satisfying $L \cdot C = 0$ takes $X$ to a point.  Then $\nu(L) = 0$.
\end{prop}

\begin{proof}
Let $(U_{i},V_{i},s_{i},w_{i})$ denote the proper surjective families of curves that define the generic quotient.  After blowing-up, we may assume that $X$ is smooth and that every $w_{i}$ is flat so that the proper quotient by the $U_{i}$ takes $X$ to a point.

Fix a very general point $x \in X$.  There is a sequence of subvarieties $T_{i} \subset X$ such that $T_{0} = x$, $T_{k} = X$, and $T_{i+1}$ is a component of $T_{i}^{\overline{w_{j}}}$ for some $j$ such that $T_{i+1}$ contains $x$.  Applying the previous lemma inductively finishes the proof.
\end{proof}

\begin{proof}[Proof of Theorem \ref{ltrivialreduction}:]
To construct the $L$-trivial reduction map, take the generic quotient (in the sense of Remark \ref{infinitefamilies}) associated to all families of movable curves $C$ with $L \cdot C = 0$ and choose a resolution $\pi: Y \to Z$ with $Y$ smooth.  Clearly this map has Property (1).    Property (3) follows from Proposition \ref{genericquotientdominance}, which says that the generic quotient is the minimal quotient as we look over all strict transform families.

Let $F$ be a general fiber of $\pi$ so that $F$ is smooth and $\phi^{*}L|_{F}$ is pseudo-effective by Theorem \ref{relativebdpp}.  Note that the families of curves $U'_{i} \times_{Y} F$ are flat over $F$ and that the proper quotient by these families contracts $F$ to a point.  By Proposition \ref{bdpptriviality}, we conclude that $\nu(\phi^{*}L|_{F})=0$, giving Property (2).
\end{proof}

\begin{exmple} \label{necessaryconditions}
Example \ref{bdppexample} gives a good illustration of the $L$-trivial reduction map.  We will keep the notation established there.  Recall that we have a connecting family of movable curves $C$ such that $L \cdot C = 0$.   To improve the situation, we need to pass to a higher model on which the strict transform family of $C$ is no longer connecting.

Let $Y$ denote a resolution of the flop, so that we have maps $\phi: Y \to X$ and $\phi': Y \to X'$ corresponding to the blow-up of $T$ and $T'$ respectively.  Consider the strict transform family of $C$ on $Y$.  This strict transform family is no longer connecting; in fact, two points are connected by the strict transform family iff they lie in a fiber of the map $p \circ \phi': Y \to \mathbb{P}^{1}$.  Thus, the generic quotient of $L$ by the family defined by $C$ is precisely the Iitaka fibration for $L$.

Note one additional feature of this example: it shows that it is necessary to work with movable curves in order to construct the $L$-trivial reduction map.  If one took the quotient of $Y$ by all families of curves with $\phi^{*}L \cdot C = 0$, one would of course recover the quotient to a point.  Similarly, we must work with irreducible families $U_{i}$, since without this condition we could reconstruct the quotient to a point by adding another component to our family.
\end{exmple}

\begin{prop} \label{ltrivialreductionproperties}
Let $X$ be normal and $L$ pseudo-effective.  Let $\pi: X \dashrightarrow Z$ denote the $L$-trivial reduction map.
\begin{enumerate}
\item If $\kappa(X,L) \geq 0$ then the Iitaka fibration for $L$ factors birationally through $\pi$.
\item If $\phi: Y \to X$ is a birational map then the $\phi^{*}L$-trivial reduction map is birationally
equivalent to $\pi$.
\item If $L$ is nef then $\pi$ is birationally equivalent to the nef reduction map.
\end{enumerate}
\end{prop}

\begin{proof}
Let $\phi: Y \to X$ be a smooth birational model that resolves the $L$-trivial reduction map.
Note that $0 \leq \kappa(F,\phi^{*}L|_{F}) \leq \nu(F,\phi^{*}L|_{F}) = 0$ for a general fiber $F$ of the $L$-trivial reduction map.  Thus the Iitaka fibration factors birationally through $\pi$, showing Property (1).

Property (2) is a consequence of the birational invariance of the generic quotient as shown in
Corollary \ref{genericquotientbirationalinvariance}.

To show Property (3), recall that the nef reduction map is the proper quotient of $X$ by all curves satisfying $L \cdot C = 0$.  Thus, the nef reduction map factors birationally through $\pi$.  Conversely, recall that $L$ is numerically trivial along a general fiber $F$ of the nef reduction map (\cite{8authors}, Theorem 2.1).  Thus $L \cdot C = 0$ for a movable curve in a very general fiber, showing that the $L$-trivial reduction map factors birationally through the nef reduction map.
\end{proof}

\section{Pseudo-effective Reduction Map} \label{pseudoeffectivereductionmapsection}

We would like our reduction map to be the maximal quotient such that the numerical dimension of $L$ vanishes along a general fiber.  It turns out that the $L$-trivial reduction map does not satisfy this property.

\begin{exmple}
There is a classical example due to Zariski of a surface $S$ carrying an irreducible curve $C$ such that $C^{2} < 0$ but $C$ has positive intersection with every other curve on $X$.  In particular the $C$-trivial reduction map is the identity map.  However $P_{\sigma}(C) = 0$, so that the map from $S$ to a point still satisfies the condition $\nu(C|_{F})=0$ on the fiber $F=S$.  Thus the $C$-trivial reduction map is not the maximal quotient with this property.
\end{exmple}

In order to improve upon the $L$-trivial reduction map, we will take the generic quotient of a different set of curves.  Since we want to focus on the vanishing of $\nu$, the most natural restriction on our curves is $\nu_{X|C}(L) = 0$.  In this section, we reinterpret the condition $\nu_{X|C}(L) = 0$ and show how to construct the pseudo-effective reduction map.  This map was first constructed in \cite{eckl05}; we hope to clarify some subtleties in the presentation.  Although our techniques are algebraic, the two approaches are ``morally'' very similar.

\subsection{The Restricted Numerical Dimension for Curves}  \label{restrictedsigmadimensionsection}

Recall that the restricted numerical dimension can be measured by taking a limit of intersections over all birational models.  For movable curves $C$, it turns out that $\nu_{X|C}(L)$ can be measured by passing to a sufficiently high birational model.

\begin{thrm} \label{stablevanishingintersection}
Let $X$ be normal and let $L$ be a pseudo-effective divisor.  Suppose that $C$ is a movable curve very general in its family.  The following are equivalent:
\begin{enumerate}
\item $\nu_{X|C}(L)=0$.
\item There is a birational map $\phi:Y \to X$ such that $P_{\sigma}(\phi^{*}L) \cdot C' = 0$ where $C'$ denotes the strict transform of $C$.
\end{enumerate}
\end{thrm}


In some sense, this statement is dual to the fact that $N_{\sigma}(L)$ has only finitely many components.

\begin{proof}[Proof of Theorem \ref{stablevanishingintersection}:]
We start with the implication (2) $\Rightarrow$ (1).  Note that (2) implies that $\nu_{Y|C'}(P_{\sigma}(\phi^{*}L)) = 0$.  (1) follows immediately from Proposition \ref{restnumdimproperties} which shows that the restricted numerical dimension is invariant under passing to birational transformations and replacing $L$ by $P_{\sigma}(L)$.

Conversely, let $w: \mathcal{C} \to X$ be a surjective generically finite morphism where $\mathcal{C}$ is a family of curves found by deforming $C$.  Let $\phi: Y \to X$ be a smooth birational map flattening $w$ and let $w': \mathcal{C}' \to Y$ denote the (flat) normalized strict transform family.  Then $w'^{-1}(\mathbf{B}_{-}(P_{\sigma}(\phi^{*}L)))$ is a countable union of subvarieties of codimension at least $2$.  Thus it does not intersect a very general curve $C'$ in our strict transform family.

Since $C'$ avoids $\mathbf{B}_{-}(P_{\sigma}(\phi^{*}L))$, we know that for any higher birational model $\psi: \widetilde{Y} \to Y$ we have
\begin{equation*}
P_{\sigma}(\psi^{*}\phi^{*}L) \cdot \widetilde{C} = P_{\sigma}(\phi^{*}L) \cdot C'
\end{equation*}
where $\widetilde{C}$ denotes the strict transform of $C'$.  Thus, the infimum of Definition \ref{restnumdimdefn} is equal to $P_{\sigma}(\phi^{*}L) \cdot C'$, proving the theorem.
\end{proof}

\subsection{Definition and Basic Properties}

We will define the pseudo-effective reduction map for a pseudo-effective divisor $L$ by taking
the generic quotient with respect to all families of movable curves whose very general member satisfies $\nu_{X|C}(L) = 0$.

\begin{proof}[Proof of Theorem \ref{psefreduction}:]
The pseudo-effective reduction map is the generic quotient
associated to all families of movable curves such that $\nu_{X|C}(L) = 0$ for
a very general member $C$.

In order to prove property (1), we will need to compare against the $L$-trivial reduction map.
First, identify a finite set of families of curves $\{ \mathcal{C}_{i} \}$ that define the generic
quotient.  By Theorem \ref{stablevanishingintersection}, for any particular family $\mathcal{C}_{i}$ there is a model $\psi_{i}: X_{i} \to X$ such that $P_{\sigma}(\psi_{i}^{*}L) \cdot C = 0$ for the strict transform of a member of $\mathcal{C}_{i}$.  Of course, this is also true for any higher model, so repeating the process we find one model $\psi: \widetilde{X} \to X$ so that $P_{\sigma}(\psi^{*}L)$ vanishes on every strict transform family.  Then the pseudo-effective reduction map is birationally equivalent to the $P_{\sigma}(\psi^{*}L)$-trivial reduction map.  Applying Theorem \ref{ltrivialreduction}, we find a model $\phi: Y \to X$ and a morphism $\pi: Y \to Z$ such that $\nu(P_{\sigma}(\phi^{*}L)|_{F}) = 0$ for the general fiber $F$ of $\pi$.  Since the restriction of $P_{\sigma}(\phi^{*}L)$ to a very general fiber is movable, Theorem \ref{numdimproperties} implies that $P_{\sigma}(\phi^{*}L)|_{F} \equiv 0$ on the very general fiber.  Thus $P_{\sigma}(\phi^{*}L)$ has vanishing intersection with any $\pi$-vertical movable curve, so Theorem \ref{relativebdpp} shows that $P_{\sigma}(\phi^{*}L)$ has vanishing intersection with any class in the subspace generated by $\overline{NM}_{1}(Y/Z)$.  Since the inclusion map for a general fiber $F$ takes $N_{1}(F)$ into this subspace (again by Theorem \ref{relativebdpp}), we obtain property (1).

Suppose that $\phi': Y' \to X$ is another birational map and $\pi': Y' \to Z'$ is another morphism with connected fibers satisfying (1).  In particular,
\begin{equation*}
\nu_{Y'|C}(\phi'^{*}L) = \nu_{Y'|C}(P_{\sigma}(\phi'^{*}L))  = 0
\end{equation*}
for a very general curve $C$ in a very general fiber $F'$ of $\pi'$.  Note that $\nu_{X|\phi'(C)}(L)=0$ as well.  Since the pseudo-effective reduction map is defined by taking the generic quotient by all such curves, $\pi$ factors birationally through $\pi'$, showing property (2).
\end{proof}

For emphasis, we extract a useful fact from the previous argument.

\begin{cor} \label{psefreductionisltrivial}
Let $X$ be normal and let $L$ be a pseudo-effective divisor.  There is some birational model $\phi: Y \to X$ such that the pseudo-effective reduction map for $L$ is birationally equivalent to the $P_{\sigma}(\phi^{*}L)$-trivial
reduction map.
\end{cor}

The pseudo-effective reduction map is also universal with respect to two other weaker properties.

\begin{thrm} \label{weakpsefreduction}
Let $X$ be a normal variety and let $L$ be a pseudo-effective divisor.  Let $\phi: Y \to X$ be a smooth model and $\pi: Y \to Z$ a morphism realizing the pseudo-effective reduction map as in Theorem \ref{psefreduction}.  Then
\begin{enumerate}
\item $\nu(P_{\sigma}(\phi^{*}L)|_{F}) = 0$ for a general fiber $F$ of $\pi$, and
\item $\nu_{X|F}(\phi^{*}L) = 0$ for a general fiber $F$ of $\pi$.
\end{enumerate}
Moreover, if $\phi': Y' \to X$ is a smooth model and $\pi': Y' \to Z'$ is a morphism with connected fibers satisfying either (1) or (2), then the pseudo-effective reduction map for $L$ factors birationally through $\pi'$.
\end{thrm}

\begin{proof}
The fact that the pseudo-effective reduction map satisfies condition (1) was shown in the proof of Theorem \ref{psefreduction}.  Note that any map satisfying condition (1) also satisfies condition (2) since
\begin{equation*}
\nu_{X|F}(\phi'^{*}L) = \nu_{X|F}(P_{\sigma}(\phi'^{*}L)) \leq \nu(F,P_{\sigma}(\phi'^{*}L)|_{F}).
\end{equation*}
Thus it suffices to see that the pseudo-effective reduction map is the maximal quotient satisfying condition (2).  Any map $\pi'$ satisfying $\nu_{X|F}(\phi'^{*}L) = 0$ also satisfies $\nu_{X|C}(\phi'^{*}L) = 0$ for a very general curve $C$ in a general fiber $F$.  We conclude by noting that the pseudo-effective reduction map is the generic quotient by all such curves.
\end{proof}

Many properties of the pseudo-effective reduction map follow from the corresponding properties for the $L$-trivial reduction map.  Property (3) is proved in \cite{eckl05} Proposition 4.5 and the others are implicit in Eckl's work.

\begin{prop} \label{psefreductionproperties}
Let $X$ be normal and let $L$ be a pseudo-effective divisor.
\begin{enumerate}
\item If $\kappa(L) \geq 0$ then the Iitaka fibration for $L$ factors birationally through
the pseudo-effective reduction map for $L$.
\item If $\mu: W \to X$ is a birational map then the pseudo-effective reduction map for $L$ is
birationally equivalent to the pseudo-effective reduction map for $\mu^{*}L$.
\item If $L$ is nef then the pseudo-effective reduction map for $L$ is birationally equivalent to
the nef reduction map.
\item The pseudo-effective reduction map for $L$ is birationally equivalent to the pseudo-effective reduction map for $P_{\sigma}(L)$.
\end{enumerate}
\end{prop}

\begin{proof}
These are all consequences of Proposition \ref{ltrivialreductionproperties} and Corollary \ref{psefreductionisltrivial}.
\end{proof}

\begin{exmple}
Suppose that $L$ has a Zariski decomposition, i.e.~there is a birational map $\phi: Y \to X$ such that $P_{\sigma}(\phi^{*}L)$ is nef.  Then the
pseudo-effective reduction map for $L$ is birationally equivalent to the nef reduction map for $P_{\sigma}(\phi^{*}L)$.
\end{exmple}

The following theorem of \cite{nakayama04} is very useful in our situation.

\begin{thrm}[\cite{nakayama04}, V.2.26] \label{relativesigmadimension0}
Let $f: X \to Z$ be a morphism of normal varieties and let $L$ be a pseudo-effective divisor on $X$ such that $\kappa(F,L|_{F}) = 0$ and  $\nu(L|_{F}) = 0$ for a general fiber $F$ of $f$.  Then there exists a morphism $g: W \to T$ of smooth varieties birationally equivalent to $f$ (with birational map $\mu: W \to X$) and a divisor $D$ on $T$ such that $P_{\sigma}(\mu^{*}L) \sim_{\mathbb{Q}} P_{\sigma}(g^{*}D)$.  If $L$ is a $\mathbb{Q}$-divisor, then $D$ is a $\mathbb{Q}$-divisor as well.
\end{thrm}

As a consequence, when $\kappa(L) \geq 0$ all the interesting geometry of $L$ can be detected on the base of the pseudo-effective reduction map:

\begin{cor}
Let $X$ be a normal variety and let $L$ be a divisor with $\kappa(L) \geq 0$.  Denote the pseudo-effective reduction map for $L$ by $\pi: X \dashrightarrow Z$.  There is a morphism $g: W \to T$ birationally equivalent to $\pi$ (with birational map $\mu: W \to X$), a divisor $D$ on $T$, and an integer $m \geq 1$ such that
\begin{enumerate}
\item $mP_{\sigma}(g^{*}D) \leq m\phi^{*}L$,
\item the degree $m$ truncations of the section rings $R(X,L)$ and $R(T,D)$ coincide, i.e.
\begin{equation*}
\bigoplus_{j} H^{0}(X,\mathcal{O}_{X}(\lfloor jmL \rfloor)) \cong \bigoplus_{j} H^{0}(T,\mathcal{O}_{T}(\lfloor jmD \rfloor)),
\end{equation*}
\item $\nu(T,D) = \nu(X,L)$.
\end{enumerate}
\end{cor}

\begin{proof}
Let $\phi: Y \to X$ be a resolution of the pseudo-effective reduction map $\pi: Y \to Z$.  Then $\nu(P_{\sigma}(\phi^{*}L)|_{F}) = 0$ and $\kappa(F,P_{\sigma}(\phi^{*}L)|_{F}) = 0$ (since $\kappa(L) \geq 0$).  Apply Theorem \ref{relativesigmadimension0} to $P_{\sigma}(\phi^{*}L)$ and $\pi$ to construct a morphism $g: W \to T$.  Since $mP_{\sigma}(\mu^{*}L) \leq m\mu^{*}L$, $R(X,L) \cong R(W,P_{\sigma}(\mu^{*}L))$, and $\nu(L) = \nu(P_{\sigma}(\mu^{*}L))$, the properties follow immediately.
\end{proof}

By analogy to the nef dimension, we define the pseudo-effective dimension $p(L)$ of $L$ to be the dimension of the image of the pseudo-effective reduction map for $L$.  Note that $p(L)$ is a numerical and birational invariant of $L$.

\begin{prop} \label{psefdimensioncomparison}
Let $X$ be smooth and let $L$ be a pseudo-effective divisor.  We have
\begin{equation*}
\kappa(L) \leq \nu(L) \leq p(L).
\end{equation*}
\end{prop}

\begin{proof}
We just need to check the last inequality.  Choose $\phi: Y \to X$ and $\pi: Y \to Z$ as in the definition of the pseudo-effective reduction map.  By Proposition \ref{numdimproperties}, we have $\nu(L) = \nu(P_{\sigma}(\phi^{*}L)) \leq \nu(P_{\sigma}(\phi^{*}L|_{F})) + \dim Z = \dim Z$ where $F$ is a general fiber of $\pi$.
\end{proof}

Note that the triviality condition $\nu(P_{\sigma}(\phi^{*}L)|_{F}) = 0$ on the fibers of the pseudo-effective reduction map is weaker than the condition $\nu(\phi^{*}L|_{F}) = 0$ on the fibers of the $L$-trivial reduction map. This distinction is crucial for the study of adjoint divisors $K_{X} + \Delta$.  One wonders whether the stronger condition $\nu(\phi^{*}L|_{F}) = 0$ holds on the fibers of the pseudo-effective reduction map as well.  Although we show it for abundant divisors in Theorem \ref{abundantequivalence}, the following example shows that this condition fails in general.

\begin{exmple} \label{fiberwise}
We will construct a threefold $X$ and a divisor $L$ such that the pseudo-effective reduction map for $L$ is a morphism $\pi: X \to S$ to a surface $S$ and $N_{\sigma}(L)$ has positive intersection with the fibers of $\pi$.  In particular, $N_{\sigma}(L)|_{F}$ is ample for a general fiber $F$ of $\pi$, so that $\pi$ does not satisfy the stronger condition $\nu(L|_{F})=0$.

Let $S$ be any surface carrying a nef divisor $D$ such that $D^{2} = 0$ and $D \cdot C > 0$ for every curve $C$ on $S$.  Fix an ample divisor $A$ on $S$.  Define $X$ to be the projective bundle
\begin{equation*}
\pi: \mathbb{P}_{S}(\mathcal{O}_{S} \oplus \mathcal{O}_{S}(-A)) \to S.
\end{equation*}
We let $T$ denote the zero section and set $L = T + \pi^{*}D$.

Our next goal is to calculate $N_{\sigma}(L)$.  \cite{nakayama04} Ch.~III.3 shows that $P_{\sigma}(L)|_{E}$ is pseudo-effective for any smooth prime divisor  $E$.  Since $(L-cT)|_{T} \equiv D+(c-1)A$ is not pseudo-effective for $c < 1$, we must have $N_{\sigma}(L) \geq T$.  But $L-T=\pi^{*}D$ is nef, so in fact $N_{\sigma}(L) = T$ and $P_{\sigma}(L) = \pi^{*}D$.  By construction $\pi^{*}D \cdot C > 0$ for any curve $C$ not contracted by $\pi$.  Thus $\pi$ is the pseudo-effective reduction map for $L$.  However, $T$ has positive intersection with the fibers of $\pi$.
\end{exmple}

One important feature of the nef reduction map is that it is almost proper, i.e.~the general fiber of $\pi: X \dashrightarrow Z$ is proper.  This is no longer true for the pseudo-effective reduction map as demonstrated by Example \ref{bdppexample}.  However, one might hope that $\mathbf{B}_{-}(L)$ represents the ``obstruction'' to the almost properness of the pseudo-effective reduction map.

\begin{ques}
Let $X$ be a smooth variety and let $L$ be a pseudo-effective divisor on $X$.  Is there a birational model $\phi: Y \to X$ centered in $\mathbf{B}_{-}(L)$ and an almost proper map $f: Y \dashrightarrow Z$ such that $f$ is birationally equivalent to the pseudo-effective reduction map?
\end{ques}

\section{Abundant Divisors} \label{abundancesection}

Recall that a divisor $L$ is said to be abundant if $\kappa(L) = \nu(L)$.  It turns out that this condition has strong geometric consequences.  As a general principle, one expects that the asymptotic behavior of sections of abundant divisors should be described by numerical properties.  This principle has been established for nef divisors in the papers \cite{kawamata85} and \cite{russo07}.  We will study the general case using the pseudo-effective reduction map.

The most important result in this section is the following list of equivalent conditions for abundance.  Partial results in this direction were known before: the most important equivalence (1) $\Leftrightarrow$ (3) is demonstrated in \cite{eckl05} and the equivalence (2) $\Leftrightarrow$ (4) was shown in \cite{nakayama04}.  The case of adjoint divisors $K_{X} + \Delta$ for a klt pair $(X,\Delta)$ was also considered by Nakayama.

\begin{thrm} \label{abundantequivalence}
Let $X$ be a normal variety and $L$ a divisor with $\kappa(L) \geq 0$.  The following are equivalent:
\begin{enumerate}
\item $\kappa(L) = \nu(L)$.
\item Let $\mu: X' \to X$ be a birational map and $f: X' \to Z'$ a morphism resolving the Iitaka fibration for $L$.  Then
\begin{equation*}
\nu(\mu^{*}L|_{F}) = 0
\end{equation*}
for a general fiber $F$ of $f$.
\item $\kappa(L) = p(L)$.
\item There is a smooth variety $W$ admitting a birational map $\mu: W \to X$ and a morphism $g: W \to T$ with connected fibers such that $P_{\sigma}(\mu^{*}L) \sim_{\mathbb{Q}} P_{\sigma}(g^{*}B)$ for some big divisor $B$ on $T$.
\end{enumerate}
\end{thrm}

It is worth emphasizing that condition (2) implies the strong condition $\nu(\phi^{*}L|_{F})=0$ on a general fiber of the pseudo-effective reduction map.

\begin{proof}
$ $

(1) $\Rightarrow$ (2): For ease of notation, note that by blowing up we may assume that the Iitaka fibration for $L$ is actually a morphism $f: X \to Z$ of smooth varieties.  We would like to show that $\nu(L|_{F}) = 0$ for a general fiber $F$ of $f$.

Recall that the numerical dimension $\nu(L)$ is defined to be the minimal dimension of a (general) reduced closed subset $W$ such that $L \not \succcurlyeq W$.  Choose a subset $W$ achieving this minimal dimension.  Note that $W$ dominates $Z$ under $f$.  If it did not, there would be some very ample divisor $H$ on $Z$ such that $f^{*}H$ contains $W$.  But since $f^{*}H \leq mL$ for some $m$, we would have $L \succcurlyeq f^{*}H \succcurlyeq W$, a contradiction.  In particular, $W \cap F$ is $0$-dimensional for a general fiber $F$ of $f$.

Let $\phi: Y \to X$ be a birational map extracting $W$ as a divisor so that $\phi^{-1}\mathcal{I}_{W} \cdot \mathcal{O}_{Y} = \mathcal{O}_{Y}(-E)$ for some divisor $E$.  Let $g$ denote the composition $f \circ \phi: Y \to Z$.  Since $W$ dominates $Z$ under $f$, the restriction of $E$ to a general fiber of $g$ has codimension $1$.

By Remark \ref{numdimalternateformulation},  $L \not \succcurlyeq W$ implies that there is an ample divisor $A$ such that $x\phi^{*}L - E + A$ is not pseudo-effective for any $x > 0$.  Now, choose any compact slice of $\overline{NM}_{1}(Y)$, and let $Q$ denote the closed region of this compact slice on which $-E + A$ is non-positive.  If $\phi^{*}L$ is positive on $Q$, then (by compactness) $x\phi^{*}L - E + A$ is positive on all of $\overline{NM}_{1}(Y)$ for $x$ sufficiently large, a contradiction.

Thus, there must be at least one curve class $\alpha \in \overline{NM}_{1}(X)$ with $E \cdot \alpha > 0$ and $\phi^{*}L \cdot \alpha = 0$.  This latter condition implies that $g^{*}H \cdot \alpha = 0$ for an ample divisor $H$ on $Z$.  Let $G$ be a general fiber of $g$.  Theorem \ref{relativebdpp} shows that there is a surjection $i_{*}: \overline{NM}_{1}(G) \to \overline{NM}_{1}(Y/Z)$.  Thus, there is some nef curve class $\beta$ on $G$ with $\phi^{*}L|_{G} \cdot \beta = 0$ and $E|_{G} \cdot \beta > 0$.  So $\phi^{*}L|_{G} \not \succcurlyeq E|_{G}$.  In other words, if $F = \phi(G)$ denotes the corresponding fiber of $f$, we have $L|_{F} \not \succcurlyeq W \cap F$.  Since $W \cap F$ has dimension $0$ this yields $\nu(L|_{F}) = 0$, finishing the implication.

\bigskip

(2) $\Rightarrow$ (3): It follows from (2) that $p(L) \leq \kappa(L)$, and the reverse inequality is shown in Proposition \ref{psefdimensioncomparison}.

\bigskip

(3) $\Rightarrow$ (4): Apply Theorem \ref{relativesigmadimension0} to the pseudo-effective reduction map to find a morphism $g: W \to T$ birationally equivalent to the Iitaka fibration such that $P_{\sigma}(\mu^{*}L) \sim_{\mathbb{Q}} P_{\sigma}(g^{*}D)$ for some divisor $D$ on $T$.  $D$ must be big for $L$ to have the correct Iitaka dimension.

\bigskip

(4) $\Rightarrow$ (1): It is clear that $P_{\sigma}(\mu^{*}L)$ is abundant.  By Theorem \ref{restnumdimproperties} this is equivalent to the abundance of $L$.
\end{proof}

\begin{cor} \label{abundantrelativelytrivial}
Let $X$ be a smooth variety, $L$ an abundant divisor on $X$, and $f: X \to Z$ a morphism birationally equivalent to the Iitaka fibration for $L$.  For any effective divisor $D \sim_{\mathbb{Q}} L$ the $f$-horizontal part $D_{hor}$ of $D$ satisfies $D_{hor} \leq N_{\sigma}(L)$.
\end{cor}

\begin{proof}
Since $\nu(D|_{F}) = \nu(L|_{F}) = 0$ for a general fiber $F$ of $f$, we know that $P_{\sigma}(D)|_{F} \equiv 0$.  But $P_{\sigma}(D)$ is effective, so that in fact $P_{\sigma}(D)|_{F} = 0$.  Thus $D_{hor} \leq N_{\sigma}(D) = N_{\sigma}(L)$.
\end{proof}

Condition (4) of Theorem \ref{abundantequivalence} implies that abundant divisors satisfy some of the same properties as big divisors.  We will collect some of these here.

\subsection{Continuity Results}

\begin{cor} \label{abundancenumerical}
Let $X$ be a normal variety and $L$ be an abundant divisor.  Suppose that $D$ is $\mathbb{Q}$-linearly equivalent
to an effective divisor and satisfies $D \equiv L$.  Then $D$ is also abundant.
\end{cor}

\begin{proof}
It suffices to show that $\kappa(D) = \kappa(L)$.  Recall that there is a birational map $\mu: W \to X$
and a morphism $g: W \to T$ with connected fibers such that $P_{\sigma}(\mu^{*}L) \sim_{\mathbb{Q}} P_{\sigma}(g^{*}B)$ for
some big divisor $B$ on $T$.  Thus  $P_{\sigma}(\mu^{*}D)|_{G} \equiv 0$ for a general fiber $G$ of $g$.  Since $D$ is $\mathbb{Q}$-linearly equivalent to an effective divisor, $P_{\sigma}(\mu^{*}D)|_{G} \sim_{\mathbb{Q}} 0$.  Thus $\kappa(G,P_{\sigma}(\mu^{*}D)|_{G}) = 0$.  Theorem \ref{relativesigmadimension0} implies that (after possibly passing to a higher birational model) $P_{\sigma}(\mu^{*}D) \sim_{\mathbb{Q}} P_{\sigma}(g^{*}B')$ for some divisor $B'$ on $T$.  By \cite{nakayama04} III.5.15 we must have $P_{\sigma}(B) \equiv P_{\sigma}(B')$.  This means that  $B'$ is a big divisor, showing the abundance of $D$.
\end{proof}

We next study the behavior of asymptotic valuations of abundant divisors.  When a valuation measures the order of vanishing along a divisor $T$ on some birational model $\phi: Y \to X$, we say that $v$ is a divisorial valuation and denote it by $v_{T}$.

\begin{prop} \label{limitofvaluations}
Let $X$ be a normal variety and let $L$ be an abundant divisor.  Suppose that $v = v_{T}$ is a divisorial valuation on $X$.  For any ample divisor $A$ on $X$
\begin{equation*}
v(\Vert L \Vert) = \lim_{m\to \infty} v\left( \left\Vert L + \frac{1}{m} A \right\Vert \right).
\end{equation*}
\end{prop}

Using Theorem \ref{abundantequivalence}, we reduce this proposition to the case where $L$ is the pullback of a big divisor.  This case is already proved in \cite{nakayama04} III.5.15.

\begin{proof}
Suppose that $\phi: Y \to X$ is a smooth birational model.  Choose an ample divisor $A_{Y}$ such that $A_{Y} - \phi^{*}A$ is ample.  Then
\begin{equation*}
\lim_{m\to \infty} v_{T}\left( \left\Vert L + \frac{1}{m} A \right\Vert \right) \geq \lim_{m\to \infty} v_{T}\left( \left\Vert \mu^{*}L + \frac{1}{m} A_{Y} \right\Vert \right).
\end{equation*}
Therefore it suffices to prove the proposition after passing to any higher model.  In particular we may assume from now on that $T$ is a divisor on $X$.

By Theorem \ref{abundantequivalence}, by passing to a higher model we may assume there is a morphism $f: X \to Z$ of smooth varieties with connected fibers such that $P_{\sigma}(\mu^{*}L) \sim_{\mathbb{Q}} P_{\sigma}(g^{*}B)$ for a big divisor $B$ on $Z$.  By flattening and resolving we may assume that $f(T)$ is a divisor $T'$ on $Z$.  \cite{nakayama04} III.1.8 states that for any pseudo-effective divisor $L$ we have
\begin{equation*}
v_{T}(\Vert L \Vert) = v_{T}(\Vert P_{\sigma}(L) \Vert) + v_{T}(N_{\sigma}(L)).
\end{equation*}
Since $\lim_{m \to \infty} v(N_{\sigma}(L+\frac{1}{m}A)) = v(N_{\sigma}(L))$, we may subtract $N_{\sigma}(L)$ from both sides of the desired equality and add on $N_{\sigma}(f^{*}B)$ to reduce to the case where $L \sim_{\mathbb{Q}} f^{*}B$.  Let $T'$ denote $f(T)$.  \cite{nakayama04} III.5.15 checks that
\begin{equation*}
\sigma_{T}(\Vert f^{*}B \Vert) = (\mult_{T}f^{*}T') \sigma_{T'}(\Vert B \Vert).
\end{equation*}
Since $B$ is big, $\sigma_{T'}(\Vert B \Vert) = v_{T'}(\Vert B \Vert)$.  We finish the proof by noting that
\begin{equation*}
v_{T}(\Vert f^{*}B \Vert) = (\mult_{T}f^{*}T') v_{T'}(\Vert B \Vert).
\end{equation*}
\end{proof}

\subsection{Positivity Results}

Abundant divisors provide a natural setting for vanishing theorems.  In general we expect that abundant divisors satisfy vanishing theorems similar in form to those for big divisors.  In this section our goal is to give an example of this principle.

\begin{prop} \label{abundantinjectivity}
Suppose that $X$ is smooth, $L$ is an abundant $\mathbb{Z}$-divisor, and $X$ admits a morphism $f: X \to Z$ that is (birationally equivalent to) the Iitaka fibration of $L$.  Suppose that $D$ is an effective $\mathbb{Z}$-divisor such that $L - \epsilon D$ is pseudo-effective for some $\epsilon > 0$ and $\Supp(D) \not \subset \mathbf{B}_{-}(L)$.  Then the map
\begin{equation*}
H^{i}(X,\mathcal{O}_{X}(K_{X} + L) \otimes \mathcal{J}(\Vert L \Vert)) \to H^{i}(X,\mathcal{O}_{X}(K_{X} + L + D) \otimes \mathcal{J}(\Vert L \Vert))
\end{equation*}
determined by $D$ is injective for every $i>0$.
\end{prop}

The main point is that the abundance of $L$ allows us to set numerical conditions on $D$.
The proof is a straightforward application of the following theorem.

\begin{thrm}[\cite{kollar95},Theorem 10.13] \label{kollarinjectivity}
Let $f: X \to Z$ be a surjective map with $X$ smooth and $Z$ normal.  Suppose that $M$ is a $\mathbb{Z}$-divisor such that $M \equiv N + \Delta$ where $N$ is the pullback of a big and nef $\mathbb{Q}$-divisor on $Z$ and $\Delta$ is a $\mathbb{Q}$-divisor with simple normal crossing support such that $\lfloor \Delta \rfloor=0$.  Let $D$ be a $\mathbb{Z}$-divisor with $D \geq 0$ such that $\Supp(D)$ does not dominate $Z$ under $f$.  Then the map
\begin{equation*}
H^{i}(X,\mathcal{O}_{X}(K_{X} + M)) \to H^{i}(X,\mathcal{O}_{X}(K_{X} + M + D))
\end{equation*}
determined by $D$ is injective for every $i>0$.
\end{thrm}

\begin{proof}[Proof of Proposition \ref{abundantinjectivity}:]
Choose $p$ sufficiently divisible so that $\mathcal{J}(\Vert L \Vert) = \mathcal{J}(\frac{1}{p} |pL|)$ and the map $\phi_{|pL|}: X \dashrightarrow T$ defined by the linear series $|pL|$ is birationally equivalent to the Iitaka fibration.  Let $\phi: Y \to X$ be a resolution of $\phi_{|pL|}$ so that $\phi^{*}|pL| = |W_{p}| + F_{p}$ for some basepoint free linear series $|W_{p}|$ and effective divisor $F_{p}$.  Let $g: Y \to T$ denote the resolved map and choose a general divisor $N \in |W_{p}|$.  Thus
\begin{equation*}
\phi^{*}L - \left \lfloor \frac{1}{p}F_{p} \right \rfloor \equiv \frac{1}{p}N + \Delta
\end{equation*}
 where $N$ is the pull-back of an ample divisor on $T$ and $\Delta$ has fractional coefficients.

The important point is to show that $\Supp(\phi^{*}D)$ does not dominate $T$ under $g$.  By assumption the restriction of $\phi^{*}(L - \epsilon D)$ to a general fiber $G$ of $g$ is pseudo-effective.   Theorem \ref{abundantequivalence} implies that $\nu(\phi^{*}L|_{G}) = 0$ for a general fiber $G$ of $g$.  In particular $P_{\sigma}(\phi^{*}L)|_{G} \equiv 0$.  Since $\nu(\phi^{*}L|_{G}) = 0$ and $\phi^{*}D|_{G}$ is effective, this can only happen if $\phi^{*}D|_{G} \leq N_{\sigma}(\phi^{*}L)|_{G}$.

If $\phi^{*}D$ had a $g$-horizontal component, the support of this component would lie in $N_{\sigma}(\phi^{*}L)$.  But this contradicts our assumption on $\Supp(D)$.  More precisely, note that by Corollary \ref{abundantrelativelytrivial} the subset $\Supp(D) \cap \mathbf{B}_{-}(L)$ does not dominate $Z$ under $f$.  Since $f$ is birationally equivalent to $g$, $\Supp(\phi^{*}D) \cap \mathbf{B}_{-}(\phi^{*}L)$ does not dominate $T$ under $g$.  We conclude that $\phi^{*}D$ has no $g$-horizontal components.

Applying Theorem \ref{kollarinjectivity} yields an injection
\begin{equation*}
H^{i}(Y,\mathcal{O}_{Y}(K_{Y} + \phi^{*}L) \otimes \mathcal{J}(\Vert \phi^{*}L \Vert)) \to H^{i}(Y,\mathcal{O}_{Y}(K_{Y} + \phi^{*}L + \phi^{*}D) \otimes \mathcal{J}(\Vert \phi^{*}L \Vert))
\end{equation*}
Pushing down to $X$, we have a map on cohomology
\begin{equation*}
H^{i}(X,\mathcal{O}_{X}(K_{X} + L) \otimes \mathcal{J}(\Vert L \Vert)) \to H^{i}(X,\mathcal{O}_{X}(K_{X} + L + D) \otimes \mathcal{J}(\Vert L \Vert))
\end{equation*}
The spectral sequence relating these cohomologies stabilizes at the first step by local vanishing (see \cite{lazarsfeld04} 9.4.1).  Thus, the map of cohomology on $X$ is still an injection.
\end{proof}

\nocite{*}
\bibliographystyle{amsalpha}
\bibliography{psefreduction}

\end{document}